\documentclass{amsart}
\usepackage{fancyhdr}
\usepackage{amsmath}
\usepackage{amsfonts}
\usepackage{amssymb}
\usepackage{amsthm,amssymb,color,comment}
\usepackage{bbm}
\usepackage{hyperref}
\hypersetup{
    colorlinks=true,
    linkcolor=black,
    filecolor=black,
    urlcolor=black,
    citecolor=black,
}
\usepackage[TS1,T1]{fontenc}
\usepackage[utf8]{inputenc}    
\usepackage{dsfont}
\usepackage{tikz}
\usepackage{enumitem}
\usepackage{geometry}
\usepackage{graphicx}
\usepackage{afterpage}
\usepackage{biblatex}
\numberwithin{equation}{section}

\newtheorem{thm}{Theorem}[section]

\newtheorem{lem}[thm]{Lemma}

\newtheorem{Ex}{Example}[section]
\newtheorem{Def}[thm]{Definition}

\theoremstyle{definition}

\newtheorem{rem}[thm]{Remark}

\DeclareMathOperator{\DIV}{div}
\DeclareMathOperator{\sign}{sign}

\newcommand{\R}{{\mathbb{R}}}
\newcommand{\rp}{{[0,\infty)}}

\newcommand{\N}{\mathbb{N}}

\newcommand{\cA}{{\mathcal{A}}}

\makeatletter

\newcommand{\Rmnum}[1]{\expandafter\@slowromancap\romannumeral #1@}
\makeatother

\title[A direct proof of existence of weak solutions to elliptic problems]{A direct proof of existence of weak solutions to fully anisotropic and inhomogeneous elliptic problems}
\author{Iwona Chlebicka, Arttu Karppinen, and Ying Li}

\address{Iwona Chlebicka\newline
Faculty of Mathematics, Informatics and Mechanics, University of Warsaw \newline
ul. Banacha 2, 02-097 Warsaw, Poland\newline
\texttt{i.chlebicka@mimuw.edu.pl}}
\address{Arttu Karppinen\newline
Faculty of Mathematics, Informatics and Mechanics, University of Warsaw \newline
ul. Banacha 2, 02-097 Warsaw, Poland\newline
\texttt{a.karppinen@uw.edu.pl}}
\address{Ying Li\newline
Department of Mathematics, Shanghai University, Shanghai 200444, China\newline
\texttt{lyinsh@shu.edu.cn}}

\thanks{{\bf Acknowledgements}: IC and AK are supported by NCN Grant Sonata Bis 2019/34/E/ST1/00120. YL is supported by China Scholarship Council, No.202106890038.}

\begin{document}

\maketitle

\begin{abstract}
    We provide a direct proof of existence and uniqueness of weak solutions to a broad family of strongly nonlinear elliptic equations with  {lower-order} terms. The leading part of the operator satisfies general growth conditions settling the problem in the framework of fully anisotropic and inhomogeneous Musielak--Orlicz spaces generated by an $N$-function $M:\Omega\times\R^d\to\rp$. Neither $\nabla_2$ nor $\Delta_2$ conditions are imposed on $M$. Our results cover among others problems with anisotropic polynomial, Orlicz, variable exponent, and double phase growth.
\end{abstract}

\section{Introduction}
In this paper we investigate the following strongly nonlinear problem
\begin{equation}\label{eq:main}
\left\{\begin{array}{cl}
-\DIV \big(\cA(x,\nabla u)+\Phi(u)\big)+b(x,u)= \DIV F &\qquad \mathrm{ in}\qquad  \Omega,\\
u(x)=0 &\qquad \mathrm{  on}\qquad \partial\Omega,
\end{array}\right.
\end{equation}
where $\Omega$ is a bounded Lipschitz domain in $ \R^d$, $d>1$.
The growth and coercivity of the vector field $\cA$ are assumed to be controlled by a generalized anisotropic $N$-function $M:\Omega\times\R^d\to\rp$, $\Phi\in L^\infty(\Omega,\R^d)$ is a~continuous vector field, $b$ is a function satisfying a sign condition without any extra growth conditions and the vector field $F$ has growth dictated by the conjugate function to $M$. The considered framework is described in~\cite{book}. We briefly recall the definitions of the involved $N$-functions, conjugates, and Musielak--Orlicz spaces in Section~\ref{sec:preliminaries}. In the present study we do not impose $\nabla_2$ nor $\Delta_2$ condition on $M$, but  we assume the modular density of smooth functions in the related Sobolev-type space that can be inferred from balanced behaviour of $M$. In this setting we show the existence of weak solutions by a direct reasoning that does not rely on a regularisation of isotropic equation. The result is supplied with a short proof of uniqueness. 

Study of nonlinear problems was initiated in 1960s by Browder~\cite{Bro}, Hess~\cite{H73}, Leray and Lions~\cite{LL} and many others. In the beginning the  considered operators were assumed to satisfy polynomial growth and coercivity conditions. The typical example of such an operator is the $p$-Laplacian $-\DIV\cA(x,\nabla u)=-\Delta_p u=-\DIV(|\nabla u|^{p-2}\nabla u)$ for $1<p<\infty$. Henceforth, the problem was stated in the reflexive setting of the classical Sobolev spaces. In recent years, there has been an increasing interest in the study of existence of solutions for elliptic problems of nonstandard growth. See~\cite{IC-pocket,over,MiRa} and references therein for a pretty complete account of new developments. Nowadays, the field of nonlinear elliptic boundary value problems aims to embrace investigations on problems admitting
\begin{itemize}
    \item general growth -- when the power function governing the growth of the operator is substituted by an $N$-function $M(x,\xi)=M (|\xi|)$, which do not necessarily satisfy the so-called $\Delta_2$-condition (being a necessary condition for an Orlicz space $L_M$ to be reflexive);
    \item inhomogeneity -- when the growth of the operator could be controlled by an $x$-dependent function e.g. $M (x, \xi) = |\xi|^{p(x)}$ (which results in the lack of the density of smooth functions in $L^{p(\cdot)}$, if $p(\cdot)$ is not regular enough);
    \item anisotropy -- when the growth of the operator is governed by a function depending on the full vector of $\xi$, not just its length $|\xi|$. For example, $M(x,\xi) = |\xi_1|^p \log(1+|\xi_1-\xi_2|)$ for $\xi = (\xi_1, \xi_2)$. 
\end{itemize} There are a lot of important ideas introduced to overcome the abovementioned challenges. The difficulty caused by the lack of reflexivity of $L_M$ under non-doubling regime was avoided by the idea of the complementary systems in Orlicz--Sobolev spaces. Contributions in this direction were initiated by Donaldson \cite{Don-71} and continued by Gossez \cite{Gos74,Gos79,Gos82} and Mustonen and Tienari \cite{MT}. We refer to \cite{ACCZG,BarCi,CN,GSG-2011} for analysis of problems in anisotropic Orlicz spaces governed by possibly fully anisotropic modular function, which is independent of the spacial variable. None of these contributions, however, provides a direct proof of existence to PDEs.  Note that the complementary systems are very homogeneous by their nature. In the typical inhomogeneous case of variable exponent problems, the setting fall into the realm of reflexive spaces, and consequently the general theory exposed e.g. in~\cite{KiSta} yields existence to the problems with regular data. 
In \cite{F12,HHK,LZ} isotropic, separable, and reflexive Musielak--Orlicz spaces are employed and \cite{DF} concerns separable, but not reflexive Musielak--Orlicz spaces. Existence to problems that are in the same time of general growth, inhomogeneous, and fully anisotropic were studied in~\cite{CGZ-19,book,dgk,GSZ-18,GSG-2008,H,LYZ,W}, but none of them provide a direct proof. Anisotropic problems with lower-order terms are less understood -- we can only refer to~\cite{DA11,GWWZ-2012}, but they do not cover our generality of the problem.
  \newline

 We are interested in a short and compact proofs of the existence and uniqueness of solutions to~\eqref{eq:main} under possibly general regime under the framework of~\cite{book}. Let us lay out our structural conditions for capturing inhomogeneous and anisotropic problems more precisely.\newline

{\bf Assumptions. } We consider a vector field $\cA:\Omega\times\R^d\to\R^d$~belonging to a Musielak--Orlicz class. Namely, we assume that a function $\cA:\Omega\times\R^d\to\R^d$ satisfies the following conditions:
\begin{itemize}
\item[(A1)] $\cA$ is a Carath\'eodory's function;
\item[(A2)] $\cA(x,0)=0$ for almost every $x \in \Omega$ and there exists an $N$-function  $M:\Omega\times\R^d\to\rp$ and  constants $c_1^\cA,c_2^\cA,c_3^\cA,c_4^\cA> 0$  such that for all $\xi\in\R^d$ we have
\[\cA(x,\xi)\cdot \xi\geq M(x,c_1^\cA\xi)-h_1(x)\]
and
\[c_2^\cA M^\ast(x,c_3^\cA \cA(x,\xi))\leq M(x,c_4^\cA\xi)+h_2(x),\]
where $M^\ast$ is the conjugate to $M$ and $h_1,h_2 \in L^1(\Omega)$;
\item[(A3)] For all $\xi,\eta\in\R^d$ and for almost every $x \in \Omega$ we have
\[(\cA(x,\xi) - \cA(x, \eta)) \cdot (\xi-\eta)\geq 0.\]
\end{itemize}
Moreover, we assume 
\begin{itemize}
\item[(P)]  $\Phi: \R \to \R^d$  {is} bounded and continuous; 
\item[(b)] $b:\Omega\times\R\to\R$ is a Carath\'eodory's function, which is {nondecreasing} with respect to the second variable, and such that  $b(\cdot,s)\in L^1(\Omega)$ and $b(\cdot,s)\sign(s)\geq 0$ for every $s\in\R$.
\end{itemize}

\medskip

We point out that we do not control the growth of $M$ with respect to the second variable by any kind of doubling condition or a power function. Instead, we rely on the density of smooth functions in a relevant function space  to study problem~\eqref{eq:main}, namely
\begin{equation*}
    V_0^1L_M(\Omega):=\{\phi\in W_{0}^{1,1}(\Omega):\ \ \nabla \phi\in L_M(\Omega)\}\,,
\end{equation*}
where $L_M(\Omega)$ is the inhomogeneous and fully anisotropic Musielak--Orlicz spaces defined in Section~\ref{sec:preliminaries}. The space $E_M$ is a closure in $L_M$-norm of the set of bounded functions. If $M$ does not satisfy growth condition of doubling type, $E_M$ is a proper subspace of $L_M$. There is a clear condition proven to be sufficient in~\cite{Bor-Chl}, but since our reasoning does not involve the condition in any other way, we assume density and recall the condition below the statement of our main theorem.  \newline

Our main theorem is an existence and uniqueness result for problem \eqref{eq:main}.

\begin{thm}\label{theo:main} Let $\Omega$ be a bounded Lipschitz domain in $\R^d$. Suppose that an $N$-function $M$ is regular enough so that the set of smooth functions is dense in $V^1_0L_M(\Omega)$  in the modular topology. Assume further that $F\in E_{M^\ast}(\Omega)$, $\cA$ satisfies assumptions (A1), (A2) and (A3), $\Phi$ satisfies (P), and $b$ satisfies (b). Then there exists at least one weak solution to the problem \eqref{eq:main}. Namely, there exists a function $u\in V^1_0L_M(\Omega)$ satisfying
\begin{equation*}
    \int_\Omega \cA(x,\nabla u) \cdot \nabla v + \Phi(u)\cdot \nabla v + b(x,u)v \, dx = \int_\Omega F \cdot \nabla v \, dx
\end{equation*}
for all $v \in V^1_0L_M(\Omega) \cap L^\infty(\Omega)$.

If, additionally, $s\mapsto b(\cdot,s)$ is strictly increasing and $\Phi$ is Lipschitz continuous, then the weak solution is unique.
\end{thm}

\begin{rem}
Our result is valid in the case of bounded data. In fact, for each $g \in L^\infty(\Omega)$, we know that there exists $F:\Omega\to \R^d$, such that $g=\DIV F$ and $F\in E_{M^\ast}(\Omega)$. The fact one can take $F\in E_{M^\ast}(\Omega)$ is a consequence of properties of Bogovski operator. This is explained in \cite[Remark~4.1.7]{book} with the use of \cite[Lemma~II.2.1.1]{Sohr}.
\end{rem} 

In order to ensure the density one may assume regularity of $M$. Note that the smooth functions are dense in $ V^1_0L_M(\Omega)$ in the modular topology if the following balance condition holds, see \cite{Bor-Chl}.

\medskip

{\bf Condition $\mathsf{(B)}$. } Given an $N$-function $M:\Omega\times \R^d\to\rp$ suppose there exists a~constant $C_M>1$ such that 
for every ball $B\subset\Omega$ with $|B|\leq 1$, every $x\in B,$ and for all $\xi\in\R^d$ such that $|\xi|>1$ and $ M(x, C_M\xi)\in[1,\tfrac{1}{|B|}]$ 
there holds $\sup_{y\in B}M(y,\xi)\leq M(x,C_M\xi)$.

\medskip

Note that in the isotropic and doubling regime this condition is known to be sufficient to the boundedness of the maximal operator. Moreover, when $d=1$ they are equivalent~\cite{Har-Has}. Condition~$\mathsf{(B)}$ is essentially less restrictive than the isotropic one from~\cite{YAGS} or the anisotropic ones used in~\cite{GSZ-18,book}. Following~\cite{Bor-Chl,Har-Has}, we give examples of $N$-functions satisfying the above balance condition.
\begin{Ex}The following $N$-functions satisfy the balance condition $\mathsf{(B)}$.
\begin{enumerate}
\item Variable exponent case: $M(x,\xi)=|\xi|^{p(x)}$,
where $p(x):\Omega \rightarrow [p^{-},p^{+}]$ is log-H\"{o}lder continuous and $1<p^{-}\leq p(\cdot)\leq p^{+}\leq \infty$; see the proof of \cite[Proposition~7.1.2]{Har-Has}.
\item Double phase case: $M(x,\xi)=|\xi|^{p}+a(x)|\xi|^{q}$, with $1< p\leq q < \infty$,  $ 0\leq a \in C^{0,\alpha}(\Omega)$,~$\alpha \in (0,1]$ and {~$\frac{p}{q}\leq 1+\frac{\alpha}{n}$}; see the proof of \cite[Proposition~7.2.2]{Har-Has}.
\item Anisotropic variable case:  $M(x,\xi)=\sum_{i=1}^{n}|\xi_i|^{p_{i}(x)}$,
where $p_{i}(x):\Omega \rightarrow [p_{i}^{-},p_{i}^{+}]$ are log-H\"{o}lder continuous and $1<p^{-}_i\leq p_i(\cdot)\leq p_i^{+}\leq \infty$; see \cite[Subsection~4.4]{Bor-Chl}. 
\item   Anisotropic double phase case: $M(x,\xi)=\sum_{i=1}^{n}\left(|\xi_i|^{p_i}+a_i(x)|\xi_i|^{  q_i}\right)$, where $1<p_i\leq q_i<\infty$, $0\leq a_i \in C^{0,\alpha_{i}}(\Omega)$, $\alpha_i\in (0,1]$, and $\frac{p_i}{q_i}\leq 1+\frac{\alpha_{i}}{n}$; for this example, as well as anisotropic multi-phase case (also with Orlicz phases), see \cite[Subsection~4.4]{Bor-Chl}.
\end{enumerate}
Taking into account~\cite{CN} and \cite[Section~4]{Bor-Chl} one can provide an explicit condition that implies $\mathsf{(B)}$ even in the case when the anisotropic function $M(x,\xi)$ does \underline{not} admit a so-called orthotropic decomposition $\sum_{i=1}^d M_i(x,\xi_i)$ even after an affine change of variables.
\end{Ex}

Let us comment on the choice of method of the proof and its importance. We point out that there are some related results in the literature. For instance~\cite{GMW} is devoted to the existence of the weak solutions to a problem of the form~\eqref{eq:main} with $\Phi\equiv 0$ and $b\equiv 0$. Since the equation is posed in the inhomogeneous and fully anisotropic space setting, where no conditions of doubling-type are assumed, the underlying space needs to be equipped with some approximation property. For this a compatibility condition might be imposed. The main idea of the proof of existence in~\cite{GMW} is to introduce a regularised problem with solutions in the classical Orlicz--Sobolev space, make use of the theory of pseudo-monotone operators due to~\cite{Gos79,GoMu,MT}, and pass to the limit. This scheme was later applied also e.g. in~\cite{ACCZG,book,GSZ-18} in the proofs of existence to problems without lower-order terms. A direct proof was a puzzling lack in this theory even in the absence of the lower-order terms.

Note that the framework of~\cite{Gos79,MT} is isotropic and describes the admitted behaviour of lower-order terms by the means of functions governing the operator. It is hard to provide natural counterpart of assumptions on the lower-order terms to get existence for possibly broad class of problems via this approach. Let us point out that no `good' embeddings of fully anisotropic Musielak--Orlicz--Sobolev spaces into  Musielak--Orlicz spaces are known. In turn, so far one is not equipped with relevant growth conditions that can be applied to the lower-order terms due to the embeddings.

Interested in the existence of solutions to as broad as possible class of problems and in a straightforward reasoning that will be relevant for the general properties of the function setting and problems with lower-order terms, we provide a direct proof of existence of weak solutions to \eqref{eq:main}. Our method is based on the Galerkin approximation. To be precise, we take a sequence of solutions to finite dimensional problems and show that it has a subsequence converging sufficiently well to our weak solution. Despite the idea is classical, the proof is not trivial because of two reasons. The first of them results from the fact that the understanding of the dual pairing is complicated ($L_{M^*}$ is not dual to $L_{M}$, but only associate).   The second reason comes from the fact that in the general Orlicz--Sobolev  spaces the set of smooth functions is not dense in the norm topology. Assuming the balanced behaviour of $M$ one is equipped with the best possible density in the framework of Musielak--Orlicz--Sobolev spaces, namely with the modular density of smooth functions, cf.~\cite{Bor-Chl}. 

Despite we follow the schemes of the proofs that can be found in~\cite{book,GWWZ-2012,Sch82}, the compactness arguments are much more delicate. We take into account  the anisotropy and inhomogeneity not considered in~\cite{Sch82}, introduce general lower-order terms unlike~\cite{book}, and relax the doubling growth condition of~\cite{GWWZ-2012}, where assuming $M^*\in\Delta_2$ is essential. In turn, most of the limit passages needed to be justified in a different way than in the above-mentioned literature. Since no theory of operator classes or regularization of the equation is used in our proof, we do not apply directly ideas of \cite{DA11,Gos74,Gos79,Gos82,MT}. This enables us to dismiss imposing growth conditions on the lower-order terms. Summing up, we provide a detailed direct proof of the existence result supplied with uniqueness shown by choosing a relevant sequence of test functions.  

\section{Preliminaries}\label{sec:preliminaries}
By $\Omega$ we always mean a bounded domain of $\R^d$ with Lipschitz regular boundary. If not specified, a constant $C$ is a positive constant, possibly changing line by line. For a measurable set $A$, we denote by  $\chi_A$ its characteristic function. 
By $C^\infty_c(\Omega)$ we mean the set of compactly supported smooth functions over $\Omega$.

Our framework is anisotropic Musielak--Orlicz spaces based on $N$-functions.  Readers are advised to turn to \cite{book} for a thorough resource. In order to introduce the main definitions let us define a~Young function in the classical way. We say a function $m:[0,\infty)\to [0,\infty)$ is a Young function if it is convex, satisfies $m(s) = 0 \Leftrightarrow s=0$ and is superlinear at zero and infinity, i.e.
\begin{equation*}
    \lim_{s\to 0^+} \frac{m(s)}{s} = 0 \quad \text{and} \quad \lim_{s\to \infty} \frac{m(s)}{s}=\infty.
\end{equation*}
A function $M(x, \xi):\Omega \times \mathbb{R}^{d}\rightarrow \mathbb{R}$  is called an $N$-function if
\begin{itemize}
\item it is a convex function with respect to ${\xi}$;
\item it is a Carath\'{e}odory function satisfying $M(x, 0)=0$; 
\item $M(x, \xi)=M(x, -\xi)$ for a.e. $x\in\Omega$;
\item there exist two Young functions $m_1, m_2: [0,\infty) \to [0,\infty)$ such that for a.e.  $x\in\Omega$ it holds
\begin{eqnarray*}
m_1(|\xi|) \leq M(x,\xi) \leq m_2(|\xi|).
\end{eqnarray*}
\end{itemize}
The conjugate function $M^*$ to  an $N$-function $M$ is defined by
\begin{equation*}
   M^\ast(x, \eta):=\displaystyle\sup_{ \xi \in \mathbb{R}^d}  \left[\xi \cdot \eta -M(x, \xi)  \right] \quad \mbox{for any}~~\eta \in \mathbb{R}^d~~ \mbox{and} ~~ \mbox{a.e.} ~~x\in \Omega. 
\end{equation*}
In that case, the Fenchel--Young inequality reads
\begin{eqnarray*}
\xi \cdot \eta \leq M(x, \xi)+M^{\ast}(x, \eta) \quad \mbox{for all}~~~\xi, \eta \in \mathbb{R}^{d} ~~~ \mbox{and}~~ \mbox{a.e.}~~x\in\Omega.
\end{eqnarray*}

For an $N$-function we define the general Musielak--Orlicz class $\mathcal{L}_{M}(\Omega)$ as the set of all measurable functions $\xi:\Omega\rightarrow \mathbb{R}^{d}$ satisfying
$$
 \int_{\Omega}M(x,\xi(x))\, dx<\infty\,.
 $$
The Musielak--Orlicz space $L_{M}(\Omega)$ is the smallest linear hull of $\mathcal{L}_{M}(\Omega)$ equipped with the Luxemburg norm
$$
\|\xi\|_{L_{M}(\Omega)}:=\inf \left\{\lambda > 0: \int_{\Omega}M \left(x,\frac{\xi(x)}{\lambda}\right)dx\leq 1 \right\}.
$$
The space $E_{M}(\Omega)$ is the closure in $L_{M}$-norm of the set of bounded functions.
Equivalently, $L_M(\Omega)$ and $E_M(\Omega)$ are defined as sets of functions $\xi:\Omega \to \R^d$ satisfying
\begin{equation*}
    \int_{\Omega}M(x,\lambda \xi(x))\, dx<\infty
\end{equation*}
for some $\lambda \in \R$ and for every $\lambda \in \R$, respectively \cite[Lemma 3.1.8]{book}. We also note that $(E_M(\Omega))^\ast = L_{M^\ast}(\Omega)$ and $(E_{M^\ast}(\Omega))^\ast = L_M(\Omega)$ \cite[Theorem 3.5.3]{book} but no other duality relations are expected.

Next we define various notions of convergence. The first is called modular convergence and it generates a more suitable topology to our possibly non-reflexive spaces than the norm topology.
A sequence $\left\{\xi_{i}\right\}^{\infty}_{i=1}$ converges modularly to ${\xi}$ in $L_{M}(\Omega)$, which we denote as ${\xi_{i}}\xrightarrow[]{M}\xi$, if
\begin{equation*}
    \int_{\Omega}M\left(x,\frac{\xi_{i}-\xi}{\lambda}\right)dx \xrightarrow{i\rightarrow \infty} 0 
\end{equation*}
for some $\lambda>0$.

Let us point out that smooth functions are not dense in our Sobolev space in the norm topology, but they are dense in the modular topology.

\begin{lem}[Theorem 1 in \cite{Bor-Chl}]\label{lem:approximation}
Assume that  $\Omega$ is  a Lipschitz domain and $M$ is an $N$-function satisfying the Balance condition ($B$). Then for any  $\phi \in V^{1}_{0}L_{M}(\Omega)$, there exists a sequence $\left\{\phi_{\delta}\right\}_{\delta>0}\in C^{\infty}_{c}(\Omega)$ satisfying  $\phi_{\delta}\rightarrow  {\phi}$ in $L^{1}(\Omega)$ and  $ \nabla\phi_{\delta} \xrightarrow{M} \nabla\phi$. Additionally, if $\phi$ is bounded, then $\|\phi_\delta\|_{L^\infty(\Omega)}\leq C(\Omega) \|\phi\|_{L^\infty(\Omega)}$ for every $\delta>0$.

\end{lem}

Again, the possible non-reflexivity forces us to consider replacements for weak convergence. Let $X$ and $Y$ be subsets of $L^1(\Omega)$ not necessarily related by duality. We say $f_n \to f$ for $\sigma(X,Y)$ if
\begin{equation*}
    \int_\Omega f_n g \, dx \xrightarrow{n\to \infty} \int_\Omega fg \, dx
\end{equation*}
for all $g \in Y$. If $X=L_M(\Omega)$ and $Y=E_{M^\ast}(\Omega)$, we recover the weak-$\ast$ convergence and can also denote $f_n \stackrel{\ast}\rightharpoonup f$.

Another consequence of modular topology is the following fact.
\begin{lem}[Lemma 3.4.6 in \cite{book}]\label{lem:modular-weak-conv}
Let $M$ be an $N$-function and $\{\xi_n\}^{\infty}_{n=1}$, $\xi \in L_M(\Omega)$. If $ \xi_n \xrightarrow{M} \xi$ in $L_M(\Omega)$ then, up to a subsequence, $\xi_n \xrightarrow{n\rightarrow \infty} \xi$ in $\sigma(L_M,L_{M^\ast})$.
\end{lem}

Next, we shall give  some definitions and preliminary lemmas related to $N$-functions and Musielak--Orlicz spaces. The first one compares a modular and a norm of a function.

\begin{lem}[Lemma~3.1.14 in \cite{book}]\label{lem:LM} 
Let $M$ be an $N$-function. 
\begin{enumerate}
    \item If $\xi \in L_M(\Omega)$ and $\|\xi\|_{L_M(\Omega)}\leq 1$, then
 $       \int_{\Omega}M(x,\xi(x))\,dx \leq \|\xi\|_{L_{M}(\Omega)}. 
  $
\item If $\xi \in L_M(\Omega)$ and $\|\xi\|_{L_M(\Omega)}>1$, then 
 $    \int_{\Omega}M(x,\xi(x))\,dx \geq \|\xi\|_{L_{M}(\Omega)}.$
\end{enumerate}
\end{lem}

Throughout the paper we take advantage of the corollary of the Vitali convergence theorem. For that we require the concept of uniform integrability.

\begin{Def}[Uniform integrability]
Let $\Omega \subset \R^d$ have finite measure. A sequence $\{f_n\} \subset L^1(\Omega)$ is uniformly integrable if 
for any $\varepsilon>0$ there exists a $\delta>0$ such that for every measurable set $E$, if $|E|<\delta$, then $\int_E |f_n| \, dx <\varepsilon$ for every $n \in \N$.

\end{Def}

 The following consequence of  in the Vitali convergence theorem results from \cite[Theorem 8.23]{Els05}. Since we assume that $\Omega$ has a finite measure,  we use pointwise convergence instead of convergence in measure.

\begin{lem}\label{lem:vitali}
If $f_n,f \in L^1(\Omega),$ $n=1,2,\dots$, and $|\Omega|< \infty$, $f_n \to f$ almost everywhere in $\Omega$ and $\{f_n\}$ is uniformly integrable in $L^1(\Omega)$, then $f_n \to f$ in $L^1(\Omega)$.
\end{lem}

Since $N$-functions are superlinear, they allow us to detect uniform integrability.
\begin{lem}[Lemma 3.4.2 in \cite{book}]\label{lem:uniformly-integrable}
Suppose $M$ is an $N$-function and let $\{\xi_n\}_{n=1}^\infty$ be a sequence of measurable functions $\xi_n: \Omega \to \R^d$ satisfying 
\begin{equation*}
   \sup_{n\in \mathbb{N}}\int_{\Omega}M(x,\xi_n(x))dx < \infty.  
\end{equation*}
Then the sequence $\{\xi_n\}_{n=1}^\infty$ is uniformly integrable in $L^1(\Omega,\R^d)$.
\end{lem}

Lastly we give a few miscellaneous and useful results. 
\begin{lem}[Lemma 8.22 in \cite{book}]\label{lem:weak-conv}
Suppose $f_k, f \in L^\infty(\Omega)$ and $g_k \rightharpoonup  g$ in $L^1(\Omega)$. Assume further that there exists a $C>0$ such that $\sup_{k \in \N}\|f_k\|_{L^\infty(\Omega)} < C$ and $f_k \to f$ almost everywhere in $\Omega$. Then
\begin{equation*}
    \lim_{k\to \infty} \int_\Omega f_k g_k \, dx = \int_\Omega fg\, dx.
\end{equation*}
\end{lem}

\begin{lem}\label{lem:PHI}
 {Let $\Phi:\R \to \R^d$ be continuous and belong to $L^\infty(\Omega,\R^d)$. Let $u \in W^{1,1}_0(\Omega)$. Then
\begin{equation*}
    \int_\Omega \Phi(u) \cdot \nabla u \, dx  = 0.
\end{equation*}}
\end{lem}

\begin{proof}
Since $\Phi$ is continuous, for each of its components $\Phi_k$ there exists a function $G_k(t) = \int_0^t \Phi_k(s) \, ds$ so that $(G_k(u))' = \Phi_k(u)\dfrac{\partial u}{\partial x_k}$ and $G_k(0)=0$. Since each $G_k$ is defined as an integral, all of them are absolutely continuous. Therefore, by denoting as $dS$ the $(n-1)$-dimensional Hausdorff measure we can use integration by parts and note that $u=0$ on $\partial \Omega$ to get
\begin{align*}
    \int_\Omega \Phi(u)\cdot \nabla u \, dx &= \sum_{k=1}^d \int_\Omega \Phi_k(u) \frac{\partial u}{\partial x_k} \, dx = \sum_{k=1}^d  \int_\Omega \dfrac{\partial}{\partial x_k}(G_k(u)) \, dx \\
    &= \sum_{k=1}^d  \left\{\int_{\partial \Omega} G_k(u) \cdot 1 \, dS - \int_{\Omega} G_k(u) \cdot 0 \, dx\right\} = 0.
\end{align*}
\end{proof}
\begin{lem}[Section 9.1 in \cite{Eva10}]\label{lem:zeros}
Let $s:\R^d\to \R^d$ be a continuous mapping and 
\begin{equation*}
    s(x)\cdot x \geq 0 \quad \mbox{if} \quad |x|=r
\end{equation*}
for some $r>0$. Then there is a point $x$ with $|x|\leq r$ such that $s(x)=0$. 
\end{lem}

\begin{lem}[Monotonicity trick, Theorem 4.1.1 in \cite{book}]\label{lem:mon-trick}
Suppose $\cA:\Omega \times \R^d \to \R^d$ satisfies growth and coercivity conditions (A1) and (A2) for an $N$-function $M:\Omega \times \R^d \to [0,\infty)$. Let $h \in L_{M^\ast}(\Omega)$ and $\xi \in L_M(\Omega)$ satisfy
\begin{equation*}
    \int_\Omega \big(h-\cA(x,\eta) \big) \cdot (\xi-\eta) \, dx \geq 0 \quad \text{for all } \eta \in \R^d.
\end{equation*}
Then $h = \cA(x,\xi)$ almost everywhere in $\Omega$.
\end{lem}


\section{Galerkin solutions}

In this section we discuss finite dimensional approximations of our problem \eqref{eq:main} and their solutions, called Galerkin solutions. In the proof of Theorem~\ref{theo:main} the existence of weak solutions is found as a limit of Galerkin solutions when the dimension of the approximating problem is increased. In the lemmas below we show the existence and boundedness properties of Galerkin solutions required for compactness arguments. The proofs follow similar scheme as in \cite[Theorem~4.1.2]{book}, but here we skip the structural assumption $M^*\in\Delta_2$ that influence the functional analysis of the underlying space, and include the analysis of lower order terms. 

Since $C^\infty_c(\Omega)$ is separable and dense in $C^1_c(\Omega)$ we can extract a sequence of $\{w_i\}_{i=1}^\infty \subset C^\infty_c(\Omega)$ such that $\overline{\text{span}\{w_1,w_2,\dots\}}^{C^1_c}=C^1_c(\Omega)$. We denote the finite dimensional spaces as $V_n := \text{span}\{w_1,\dots w_n\}$.

\begin{Def}\label{def:Galerkin}
For every $n=1,2, \dots$ and spaces $V_n$ defined above, we say a function $u_n \in V_n$ is called a Galerkin solution to the problem \eqref{eq:main} if 
\begin{equation}\label{eq:Galerkin}
        \int_\Omega \cA(x,\nabla u_n)\cdot \nabla \omega +\Phi(u_n)\cdot \nabla \omega + b(x,u_n)\omega\, dx = \int_\Omega F\cdot  \nabla \omega \ dx.
    \end{equation} 
holds for every $\omega \in V_n$.
\end{Def}

\begin{lem}[Existence of Galerkin solutions]\label{lem:Galerkinsolution}
Let $\Omega \subset \R^d$ be a bounded Lipschitz domain. Assume $\cA$ satisfies (A1), (A2) and (A3) and $M:\Omega \times \R^d \to [0,\infty)$ is an $N$-function.  Let  $F\in E_{M^\ast}(\Omega)$, $b: \Omega\times \R \to \R$  satisfies condition (b) and $\Phi:\R \to \R^d$ satisfy (P). Then for every $n \in \N$, there exists a Galerkin solution in the sense of \eqref{eq:Galerkin}.

\end{lem}

\begin{proof}
Every $u_n \in V_n$ can be written in a form $u_n = \sum_{k=1}^n \alpha_k w_k$ for $n \in \N$, where $\alpha_k \in \R$. We show that there exist such $\{\alpha_k\}_{k=1}^n$ such that 
\begin{equation*}
    \int_\Omega \cA(x,\nabla u_n) \cdot \nabla w_k + \Phi(u_n)\cdot \nabla w_k+ b(x,u_n)w_k \, dx = \int_\Omega F\cdot \nabla w_k \, dx.
\end{equation*}

We define $\vec{\alpha} = (\alpha_1,\alpha_2, \dots \alpha_n) \in \R^n$, a mapping $s:\R^n \to \R^n$ given coordinate by  the formula
\begin{equation*}
    s_j(\vec\alpha) := \int_\Omega \cA\left(x, \sum_{k=1}^n \alpha_k \nabla w_k\right) \cdot \nabla w_j +\Phi\left(\sum_{k=1}^n \alpha_k  w_k\right)\cdot \nabla w_j +b\left(x,\sum_{k=1}^n \alpha_k  w_k\right)w_j  - F\cdot \nabla w_j \, dx
\end{equation*}
and \[\omega(\vec\alpha) :=\sum_{k=1}^n \alpha_k w_k.\] In order to prove that $s$ takes the value zero, first we show that $s$ is continuous and for that we choose a sequence $\{\vec a^i\}_{i=1}^\infty$ such that $\vec a^i \to \vec\alpha$ in $\R^n$. Consider
\begin{equation*}
\begin{split}
      |s_j(\vec a^i) - s_j(\vec\alpha)| =&\bigg|\int_{\Omega} h_{i,j}(x)\,dx\bigg| 
\end{split}
\end{equation*}
for all $j=1, \dots, n$ and every $i \in \N$, where 
\begin{equation*}
    h_{i,j}(x) := \big(\cA(x,\nabla \omega(\vec a^i)) - \cA(x,\nabla \omega(\vec \alpha))\big) \cdot \nabla w_j + \big(\Phi(\omega(\vec a^i))-\Phi(\omega(\vec \alpha))\big) \cdot \nabla w_j + \big(b(x, \omega(\vec a^i)) - b(x,\omega(\vec \alpha))\big) w_j.
\end{equation*}
Since $\cA$ and $b$ are Carath\'eodory functions and $\Phi$ is continuous, we see that for almost every $x \in \Omega$ we have $h_{i,j}(x)\to 0$  as $i \to \infty$. We show uniform integrability of the three terms in $h_{i,j}$ one by one starting with $\cA$. Assumption (A2) shows that
\begin{equation*}
    c_2M^\ast(x,c_3\cA(x,\nabla \omega(\vec a^i))) \leq M(x,c_4 \nabla \omega(\vec a^i)) + h_2(x)
\end{equation*}
and for $\vec a^i=(a^i_1,\dots, a^i_n)$
\begin{align*}
   \int_\Omega M(x,c_4 \nabla \omega(\vec a^i)) \, dx &\leq \sum_{k=1}^n \frac{|a^{i}_k|}{|\vec a^i|} \int_\Omega M(x,c_4|\vec a^i|\nabla w_k) \, dx \leq n \max_{k\in\{1,\dots,n\}} \int_\Omega M(x,c_4 |\vec a^i|\nabla w_k) \, dx 
\end{align*}
and this is finite since $\{\vec a^i\}_{i=1}^\infty$ is bounded $\R^n$ and $w_k\in C_c^\infty(\Omega)$, see \cite[Lemma~3.1.8]{book}. Now Lemma~\ref{lem:uniformly-integrable} shows the uniform integrability of $\{\cA(\cdot, \nabla \omega(\vec a^i))\}_{i=1}^\infty$ in $L^1(\Omega,\R^d)$. As $\cA(\cdot, \nabla \omega(\vec \alpha)) \in L^1(\Omega,\R^d)$ and $\nabla w_k \in L^\infty(\Omega,\R^d)$ and it follows that $\big(\cA(x,\nabla \omega(\vec a^i)) - \cA(x,\nabla \omega(\vec \alpha))\big) \cdot \nabla w_j$ is uniformly integrable in $L^1(\Omega)$.

Moreover, we notice that $|\Phi(\cdot)|\in L^\infty(\R)$ and $|\nabla w_k| \in L^\infty(\Omega)$. Thus for any set $E\subset \Omega$
\begin{equation*}
    \left |\int_E \Phi(\omega( \vec a^i)) \cdot \nabla w_k \, dx \right| \leq \max_{k \in \{1,\dots, n\}}  \int_E \||\Phi(\omega( \vec a^i))|\|_{L^\infty(\R)} \||\nabla w_k|\|_{L^\infty(\Omega)} \, dx < C |E|.
\end{equation*}
Furthermore by the fact that  $s\mapsto b(\cdot,s)$ is non-decreasing, convergence of $\vec a^i$ and boundedness of $\{w_k\}$ we get
\begin{align*}
    \left|\int_E \big(b(x, \omega(\vec a^i))\big) w_k \, dx \right|\leq   n \max_{k\in\{1,\dots,n\}} \left\{ \|w_k\|_{L^\infty(\Omega)} \int_E |b(x, |\vec a^i|\,w_k)|\, dx \right\}< C|E|.
\end{align*}
Thus all the terms in $h_{i,j}$ are uniformly integrable in $L^1(\Omega)$ and $h_{i,j}\to 0$ almost everywhere, so it follows from Vitali convergence theorem (Lemma~\ref{lem:vitali}) that
\begin{equation*}
    |s(\vec a^i)-s(\vec \alpha)| \leq n \max_{j\in\{1,\dots,n\}} \int_\Omega |h_{i,j}| \, dx \to 0
\end{equation*}
as $i \to \infty$. Thus $s$ is continuous.

Our next aim is to show that there exists a vector $\vec \alpha\in \R^n$ such that $s(\vec \alpha) = 0$. For this we will apply  Lemma \ref{lem:zeros}. Therefore we need to prove that  $s(\vec \alpha)\cdot \vec \alpha \geq 0$ if $|\vec \alpha|=R$ for some $R>0$. From (A2), Young's inequality,  the facts that $\Phi(\omega(\vec \alpha))\cdot \nabla \omega(\vec \alpha) =0$ and $b$ satisfies (b), and Lemma \ref{lem:LM}(2) we get
\begin{align*}
    s(\vec \alpha)\cdot \vec \alpha &=\int_\Omega \cA(x,\nabla \omega(\vec \alpha)) \cdot \nabla \omega(\vec \alpha)  + b(x,\omega(\vec \alpha))\omega(\vec \alpha)- F\cdot \nabla \omega(\vec \alpha) \, dx \\
    &\geq \frac{1}{2} \int_\Omega M(x,c_1^\cA \nabla \omega(\vec \alpha)) \, dx  
    - h_1(x)\, dx- \int_\Omega M^\ast\left(x, \frac{2}{c_1}F\right ) \, dx \\
    &\geq \frac{1}{2} (c_1^\cA \|\nabla \omega(\vec \alpha)\|_{L_M(\Omega)} -1)- C   
    .
\end{align*}
Note that $\vec \alpha \mapsto \|\nabla \omega(\vec \alpha)\|_{L_M(\Omega)}$ is a continuous function and therefore it attains its minimum on the unit sphere $S_1 \subset \R^n$ at some $\vec \beta=(\beta_1,\dots,\beta_n) \in S_1$. 

We show that $\|\nabla \omega(\vec \beta)\|_{L_M(\Omega)} >0$. Assume the contrary that $\|\nabla \omega(\vec \beta)\|_{L_M(\Omega)} =0$. Thus $\|\sum_{k=1}^n \beta_k \nabla w_k\|_{L^1(\Omega)} =0$ and by Poincar\'e inequality we also have $\|\sum_{k=1}^n \beta_k w_k\|_{L^1(\Omega)} =0$ (recall that $w_k \in C_c^\infty(\Omega)$). This implies that $\sum_{k=1}^n \beta_k w_k = 0$ almost everywhere in $\Omega$, i.e. $\beta_k=0$ for each $k = 1, \dots, n$ since functions $w_k$ are linearly independent. Since $\vec \beta \in S_1$ this is contradiction so $\|\nabla \omega(\vec \beta)\|_{L_M(\Omega)} >0$. 

For our choice of $\vec \beta$ we have
\begin{equation*}
    \|\nabla \omega(\vec \alpha)\|_{L_M(\Omega)}\geq  |\alpha| \left \| \nabla \omega\left(\frac{\vec \alpha}{|\vec \alpha|}\right) \right\|_{L_M(\Omega)} \geq |\vec \alpha| \|\nabla \omega (\vec \beta)\|_{L_M(\Omega)}>0.
\end{equation*}
This means that $\|\nabla \omega(\vec \alpha)\|_{L_M(\Omega)} \to \infty$ when $|\vec \alpha| \to \infty$. Thus $s(\vec \alpha)\cdot \vec \alpha \geq 0$ if $|\vec \alpha|=R$ for some $R>0$.  From the definition of $s$ we see that for  $u_n=\sum_{k=1}^n\alpha_k w_k$ we have
\begin{equation*}
    \int_\Omega \cA(x,\nabla u_n) \cdot \nabla \omega+ \Phi(u_n)\cdot \nabla \omega + b(x,u_n) \omega \, dx = \int_\Omega F\cdot \nabla \omega \, dx
\end{equation*}
and thus the Galerkin solution $u_n$ exists.
\end{proof}

\begin{lem}[Uniform boundedness of Galerkin solutions]\label{lem:Galerkin-bounded}
Let $\Omega \subset \R^d$ be a bounded Lipschitz domain. Assume $\cA$ satisfies (A1), (A2), (A3) and $M:\Omega \times \R^d \to [0,\infty)$ is an $N$-function.  Let  $F\in E_{M^\ast}(\Omega)$, $b: \Omega\times \R \to \R$  satisfy condition (b) and $\Phi:\R \to \R^d$ satisfy (P). Then there exists a constant $C$ independent of $n$ such that for every Galerkin solution $u_n$ it holds 
 \begin{equation*}
    \int_\Omega \cA(x,\nabla u_n) \cdot \nabla u_n \, dx \leq C; \quad \|\nabla u_n\|_{L_M(\Omega)} \leq C; \quad \int_\Omega b(x,u_n)u_n \, dx \leq C.
\end{equation*}
\end{lem}

\begin{proof}
We test the equation \eqref{eq:Galerkin} by $u_n$, which is an admissible test function due to Lemma~\ref{lem:Galerkinsolution}, to get
\begin{equation*}
    \int_\Omega \cA(x,\nabla u_n) \cdot \nabla u_n + \Phi(u_n)\cdot \nabla u_n + b(x,u_n)u_n \, dx = \int_\Omega F \cdot \nabla u_n \, dx.
\end{equation*}
Using condition (A2) we get an estimate
\begin{equation*}
    \frac{1}{2}\int_{\Omega} M(x,c_1^\cA\nabla u_n) -h_1(x)\,dx \leq \frac{1}{2}\int_{\Omega}A(x,\nabla u_n)\cdot \nabla u_n \, dx.
\end{equation*}
Moreover, Fenchel--Young inequality and the definition of $N$ functions allow us to infer that
\begin{align*}
    \int_{\Omega}\frac{4}{c_1^\cA}F\cdot \frac{c_1^\cA}{4}\nabla u_n dx &\leq \int_{\Omega} M\left(x,\frac{c_1^\cA}{4}\nabla u_n\right)\, dx +\int_{\Omega} M^\ast \left(x,\frac{4}{c_1^\cA}F\right)\, dx  \\
    &\leq  \frac{1}{4} \int_{\Omega} M(x,c_1^\cA \nabla u_n) dx +\int_{\Omega} M^\ast \left(x,\frac{4}{c_1^\cA}F\right)\, dx.
\end{align*}
And thus, according to  $M(x,\xi) \geq 0$ for almost every $x \in \Omega$ and every $\xi \in \R^d$, $b(x,u_n)u_n \geq 0$ as $b$ satisfies condition (b), and $\Phi(u_n)\cdot \nabla u_n = 0$ by Lemma~\ref{lem:PHI} we have
\begin{align*}
    &\frac{1}{4} \int_{\Omega} M(x,c_1^\cA\nabla u_n) dx + \frac{1}{2}\int_{\Omega}\cA(x,\nabla u_n)\cdot \nabla u_n  dx + \int_\Omega \Phi(u_n)\cdot \nabla u_n dx +\int_{\Omega}b(x,u_n)u_n \, dx \\
    &\quad \leq \int_{\Omega} M^\ast \left(x,\frac{4}{c_1^\cA}F\right ) +\frac{1}{2}\|h_1(x)\|_{L^1(\Omega)}.
\end{align*}
Since $F\in E_{M^\ast}(\Omega)$ by assumption, the right-hand side of the latter inequality is finite and we infer that 
\begin{equation*}
 \int_\Omega \cA(x,\nabla u_n)\cdot \nabla u_n \, dx \leq C
\end{equation*}
and
\begin{equation*}
    \int_\Omega b(x,u_n) u_n \, dx \leq C.
\end{equation*}
Furthermore, by Lemma~\ref{lem:LM}
\begin{equation*}
    \|\nabla u_n\|_{L_M(\Omega)} \leq \frac{1}{c_1^\cA}\left( \int_\Omega M(x,c_1^\cA \nabla u_n) \, dx + 1 \right) \leq C,
\end{equation*}
where $C$ is independent of $n$. 
\end{proof}
The proof of the next fact follows almost the same lines as \cite[Lemma 3.8.2]{book}, but we correct some flaws and  include it for readers' convenience.

\begin{lem}\label{lem:cA-bounded}
Suppose $M$ is an $N$-function and $\cA:\Omega \times \R^d \to \R^d$ satisfies (A1), (A2), (A3) and suppose $\|\cA(\cdot, \xi)\cdot \xi\|_{L^1(\Omega)} \leq \tilde c$. Then there exists a constant $C>0$ depending only on the parameters from (A1), (A2) and $\tilde c$ such that
$\|\cA(\cdot, \xi)\|_{L_{M^\ast}(\Omega)} < C$.
\end{lem}

\begin{proof}
Let $w \in E_M(\Omega)$. Since $\cA$ is a Carath\'eodory's function, we see that $\cA(x,w)$ is bounded and thus belongs to $E_{M^\ast}(\Omega)$. Now by a direct calculation and (A3)
\begin{align*}
    \int_\Omega \cA(x,\xi) \cdot w \, dx  &= -\int_\Omega (\cA(x,\xi) - \cA(x,w)) \cdot (\xi -w) \, dx + \int_\Omega \cA(x,\xi) \cdot \xi\, dx \\
    &\quad - \int_\Omega \cA(x,w)\cdot (\xi-w) \, dx \\
    &\leq \int_\Omega \cA(x,\xi) \cdot \xi\, dx - \int_\Omega \cA(x,w)\cdot (\xi-w) \, dx.
\end{align*}
The first term on the right-hand side is bounded by assumption.
For the second term we use Fenchel--Young inequality, (A2) and convexity to estimate 
\begin{align*}
    -\int_\Omega \cA(x,w)\cdot (\xi-w) \, dx &= -\frac{2}{c_1^\cA c_3^\cA}\int_{\Omega}c_3^\cA \cA(x,w)\cdot  \left(\frac{c_1^\cA(\xi-w)}{2}\right)\,dx\\ 
    & \leq \frac{2}{c_1^\cA c_3^\cA}\int_{\Omega} {M^\ast}(x,c_3^\cA \cA(x,w)) +M\left(x,\frac{c_1^\cA (\xi-w)}{2}\right)\,dx\\
    & \leq \frac{2}{c_1^\cA c_3^\cA} \int_{\Omega}\frac{1}{c_2^\cA}  M(x,c_4^\cA w) +\frac{1}{c_2^\cA} h_2(x) +M(x,c_1^\cA \xi)+M(x,c_1^\cA w)\,dx\\
    &\leq \frac{2}{c_1^\cA c_3^\cA}\left[ \left(\frac{1}{c_2^\cA}+1\right)\int_{\Omega}M(x,\max\{{c_1^\cA,c_4^\cA}\}w)\,dx\right.\\
    &\left.\qquad\qquad\ +\int_{\Omega} \cA(x,\xi) \cdot\xi+ h_1(x)+\frac{1}{c_2^\cA} h_2(x)\,dx\right].
 \end{align*}
In turn, 
\begin{align}
\nonumber    \int_\Omega \cA(x,\xi) \cdot w \, dx\leq  \frac{2 }{c_1^\cA c_3^\cA} &\left[\left(\frac{1}{c_2^\cA}+1\right)\int_{\Omega}M(x,\max\{{c_1^\cA,c_4^\cA}\}w)\,dx\right. \\
    &\left.\ +\left(\frac{c_1^\cA c_3^\cA}{2}+1\right)\int_\Omega \cA(x,\xi) \cdot \xi + h_1(x)+\frac{1}{c_2^\cA}h_2(x)\, dx\right]\,.\label{est}
\end{align}
Provided $\max\{{c_1^\cA,c_4^\cA}\}\|w\|_{L_M(\Omega)}\leq 1$, by Lemma~\ref{lem:LM}(1), we know that 
\begin{align*}
    \int_{\Omega}M(x,\max\{{c_1^\cA,c_4^\cA}\}w)\,dx \leq \max\{{c_1^\cA,c_4^\cA}\}\|w\|_{L_M(\Omega)}\leq 1.
\end{align*}
Using it in~\eqref{est} we obtain
\begin{align*}
    \|\cA(x,\xi)\|_{(E_M(\Omega))^\ast} &= \max\{c_1^\cA,c_4^\cA\}\sup\left\{\int_{\Omega}\cA(x,\xi)\cdot  w\,dx:\quad \max\{c_1^\cA,c_4^\cA\}\|w\|_{L_M(\Omega)}\leq 1 \right\} \\
    &\leq \frac{2\max\{c_1^\cA,c_4^\cA\}}{c_1^\cA c_3^\cA} \left[\left(\frac{1}{c_2^\cA}+1\right) +\left(\frac{c_1^\cA c_3^\cA}{2}+1\right)\int_\Omega \cA(x,\xi) \cdot \xi + h_1(x)+\frac{1}{c_2^\cA}h_2(x)\, dx\right] \\
    &\leq C.
\end{align*}
Thus $\cA(x,\xi)$ stays bounded in $L_{M^\ast}(\Omega)$.
\end{proof}

\section{Existence and uniqueness for problem (\ref{eq:main})}

We are ready to prove our main Theorem.



\begin{proof}[Proof of Theorem~\ref{theo:main}]
We base our proof on Galerkin approximations.  We fix spaces $V_n$ and functions $u_n$ as in Definition~\ref{def:Galerkin}.
We shall prove that there exists a weak solution of the problem \eqref{eq:main} being a limit of a subsequence of the Galerkin solutions $\{u_n\}$. We divide our proof into four steps.\newline

\textbf{Step 1} According to the Lemma~\ref{lem:Galerkinsolution}, we get
\begin{equation}\label{eq:galerkin2}
    \int_\Omega \cA(x,\nabla u_n)\cdot \nabla \omega +\Phi(u_n)\cdot \nabla \omega+ b(x,u_n)\omega\, dx = \int_\Omega F \cdot \nabla \omega \, dx\quad\text{
for all $\omega \in V_k$, $k\leq n$.}
\end{equation} By Lemma~\ref{lem:Galerkin-bounded} we have that $\|\nabla u_n\|_{L_M(\Omega)} \leq C$ and $\|\cA(\cdot,\nabla u_n)\cdot \nabla u_n\|_{L^1(\Omega)}\leq C$, where $C$ is independent of $n$. 
By the Banach--Alaoglu Theorem, we know that $\{\nabla u_n\}_n$ is weakly-$\ast$ compact in $L_M(\Omega)$. By Lemma~\ref{lem:uniformly-integrable} and the fact that $M$ is an $N$-function, we see that $\{\nabla u_n\}_n$ is uniformly integrable in $L^1(\Omega,\R^d)$. The Dunford--Pettis Theorem \cite[Theorem 4.3]{Bre11} implies that $\left\{\nabla u_{n}\right\}_{n}$ is relatively compact in the weak topology of $L^1(\Omega,\R^d)$. Taking into accont the Rellich-Kondrashev theorem and Poincar\'e inequality there exists a function $u \in W^{1,1}_0(\Omega)$ such that \[\text{$u_n \rightharpoonup u\quad$ in $\quad W^{1,1}(\Omega)$.}\] Since the weak-$\ast$ limit is unique, we see that additionally $\nabla u \in V^1_0L_M(\Omega)$  and \[\text{$\nabla u_n \stackrel{*}\rightharpoonup  \nabla u\quad$ for $\quad\sigma(L_M,E_{M^\ast})$.}\] Also, by Lemma~\ref{lem:cA-bounded}, we see that $\cA(\cdot,\nabla u_n)$ is bounded in $L_{M^\ast}(\Omega)$ and a similar argument shows the existence of $h \in L_{M^\ast}(\Omega)$ such that
\begin{equation}\label{limA}
    \cA(x,\nabla u_n)\stackrel{*}\rightharpoonup  h\quad \text{ for }\quad \sigma(L_{M^\ast}, E_M).
\end{equation}


As $b$ is a Carath\'eodory's function, we get $b(x,u_n) \to b(x, u)$ almost everywhere in $\Omega$. Fix $\varepsilon>0$ and choose $\ell$ large enough to guarantee $\frac{C}{\ell} < \frac{\varepsilon}{2}$ for $C$ from  Lemma~\ref{lem:Galerkin-bounded}. Since $b(\cdot,\ell)\in L^1(\Omega)$ is a given function independent of $n$, for any subset $E\subset \Omega$ with $|E|$ small enough, we have $\|b(\cdot,\ell)\|_{L^1(E)} < \frac{\varepsilon}{2}$. Now
\begin{align*}
    \int_{E} b(x,u_n) \, dx &= \int_{E\cap \{|u_n|<
    \ell\}} b(x,u_n) \, dx + \int_{E\cap \{|u_n| \geq \ell\}} b(x,u_n) \, dx \\
    &\leq \int_{E\cap \{|u_n|<\ell\}} |b(x,\ell)| \, dx + \int_{E\cap \{|u_n| \geq \ell\}} \frac{1}{|u_n|} |b(x,u_n) u_n| \, dx \\
&\leq \int_{E} |b(x,\ell)| \, dx + \int_{\{|u_n| \geq \ell\}} \frac{1}{\ell} |b(x,u_n) u_n| \, dx \\
     &\leq \int_{E} |b(x,\ell)| \, dx + \frac{C}{\ell} < \varepsilon.
\end{align*}
  Therefore $\{b(\cdot,u_n)\}$ is uniformly integrable.
By Vitali convergence theorem (Lemma~\ref{lem:vitali}) we have
\begin{equation}\label{limb}
   b(\cdot, u_n) \to b(\cdot, u) \text{ in } L^1(\Omega). 
\end{equation}
As $\Phi$ is continuous and belongs to $L^\infty(\Omega,\R^d)$, we immediately see that \begin{align}
    \label{limPhi} \text{$\Phi(u_n) \to \Phi(u)$ weakly in $L^1(\Omega,\R^d)$}
\end{align} due to dominated convergence theorem.

Having \eqref{limA}, \eqref{limPhi}, and \eqref{limb}, we can pass to the limit in~\eqref{eq:galerkin2} with $n$. In turn we get
\begin{equation}\label{eq:step2-1}
    \langle F, \nabla \omega \rangle = \int_{\Omega} h \cdot \nabla \omega +\Phi(u)\cdot \nabla \omega + b(x,u)\omega \, dx\quad\text{for any $\omega \in V_k$, $k\in\N$}.
\end{equation}

\textbf{Step 2} (extending the class of test functions): Let $\varphi \in C_c^\infty(\Omega)$ be arbitrary and $\varphi_j \in V_j$ be a sequence of smooth function such that $\varphi_j \to \varphi$ in $C^1_c(\Omega)$. {Due to~\eqref{eq:step2-1} we have 
\begin{align}\label{ngeqj}
\begin{split}
    \langle F, \nabla \varphi_j \rangle &= \int_{\Omega} 
    h\cdot \nabla \varphi_j +\Phi(u)\cdot \nabla \varphi_j + b(x,u)\varphi_j \, dx\quad\text{for any $j$.}
\end{split}
\end{align}}

Now $h \cdot \nabla \varphi_j \to h \cdot \nabla \varphi$ almost everywhere and by uniform convergence of $\{\nabla \varphi_j\}$ we have for large enough $j$
\begin{equation*}
    \int_\Omega  h \cdot \nabla \varphi_j \ dx \leq \int_\Omega h \cdot \nabla \varphi + 1 \, dx < \infty.
\end{equation*}
Since~\eqref{ngeqj}, $\Phi(u)$ is bounded, $b(\cdot,u)\in L^1(\Omega)$, by dominated convergence theorem applied similarly also to all the other terms in  we have
\begin{align}\label{eq:h-for-smooth2}
\begin{split}
   \langle F, \nabla \varphi \rangle = \lim_{j \to \infty} \langle F, \nabla \varphi_j\rangle  &= \lim_{j\to \infty} \int_\Omega h \cdot \nabla
   \varphi_j+\Phi(u)\cdot \nabla \varphi_j+ b(x,u) \varphi_j \, dx  \\
   &= \int_\Omega h \cdot \nabla \varphi +\Phi(u)\cdot \nabla \varphi + b(x,u) \varphi \, dx.
\end{split}
\end{align}

Let $v \in V^1_0L_M(\Omega)\cap L^\infty(\Omega)$ and $\{v_k\} \subset C_c^\infty(\Omega)$ be a sequence such that $\nabla v_k \to \nabla v$ modularly in $L_M(\Omega)$, $v_k \to v$ in $L^1(\Omega)$ and $\|v_k\|_{L^\infty(\Omega)} \leq c \|v\|_{L^\infty(\Omega)}$ with $c$ independent of $k$. This is possible because of the modular density assumption ~(Lemma~\ref{lem:approximation}). 

The modular convergence of $\{\nabla v_k\}$  implies that $\nabla v_k \to \nabla v$ for $\sigma(L_M, L_{M^\ast})$ by Lemma \ref{lem:modular-weak-conv}. Note that
\begin{equation}\label{eq:lots-of-info}
\text{$F \in E_{M^\ast}(\Omega) \subset L_{M^\ast}(\Omega)$, $\quad h \in L_{M^\ast}(\Omega)$, $\quad\Phi(u)\in L^\infty(\Omega,\R^d)\subset E_{M^\ast}(\Omega)$}\end{equation}  
and since $b(\cdot, u) \in L^1(\Omega)$, Lemma \ref{lem:weak-conv} implies that 
\begin{equation}\label{eq:galerkin3}
    \lim_{k\to \infty}\int_\Omega b(x,u) v_k \, dx \to \int_\Omega b(x,u)v \, dx. 
\end{equation} Using~\eqref{eq:lots-of-info} and~\eqref{eq:galerkin3} we can extend \eqref{eq:h-for-smooth2} as
\begin{align}\label{eq:h-for-sobolev2}
\begin{split}
    \langle F, \nabla v\rangle = \lim_{k \to \infty} \langle F,\nabla v_k \rangle &=\lim_{k \to \infty} \int_\Omega h \cdot \nabla v_k +\Phi(u)\cdot \nabla v_k+b(x,u)v_k \, dx \\
    &= \int_\Omega h \cdot \nabla v+\Phi(u)\cdot \nabla v +b(x,u) v \, dx 
\end{split}
\end{align}
for test functions $v \in V^1_0L_M(\Omega) \cap L^\infty(\Omega)$.\newline

\textbf{Step 3} ($h = \cA(x,\nabla u)$ almost everywhere): Let $w \in L^\infty(\Omega,\R^d)$ be arbitrary. By (A3) we have
\begin{align*}
    0 &\leq \limsup_{n \to \infty} \int_{\Omega} (\cA(x,\nabla u_n)  - \cA(x,w)) \cdot (\nabla u_n - w) \, dx \\
    &= \limsup_{n\to \infty}\bigg(\int_{\Omega} \cA(x,\nabla u_n) \cdot \nabla u_n \, dx - \int_\Omega \cA(x,\nabla u_n)\cdot w \, dx \\
    &\qquad \qquad  - \int_{\Omega} \cA(x,w) \cdot \nabla u_n \, dx + \int_{\Omega} A(x,w)\cdot w \, dx\bigg) \\
    &\leq I_1 + I_2 + I_3 + I_4.
\end{align*}
Let us handle the integrals one by one.  Before we estimate the $I_1$, we would like to point out that, as $u_n\to u$ a.e. in $\Omega$, $b:\Omega\times\R\to\R$ is a Carath\'eodory's function
satisfying the sign condition, using the Fatou lemma, we have 
\[\int_{\Omega}b(x,u)u\,dx\leq \liminf_{n\to \infty}\int_{\Omega}b(x,u_n)u_n\,dx\leq C.\]

In the case of $I_1$,  as $u_n$
is a Galerkin solution, $F \in E_{M^\ast}(\Omega)$, taking the test function as $u_n$ in  \eqref{eq:Galerkin} and combining with Lemma \ref{lem:PHI}, we deduced that

\begin{align*}\label{eq:I_1}
\begin{split}
I_1 &= \limsup_{n \to \infty} \int_\Omega \cA(x,\nabla u_n) \cdot \nabla u_n \, dx\\
&=\limsup_{n \to \infty} \int_{\Omega}\bigg(F \cdot \nabla u_n-b(x,u_n)u_n-\Phi(u_n)\cdot \nabla u_n\bigg)\,dx\\
&\leq \limsup_{n \to \infty} \int_\Omega F \cdot \nabla u_n \,dx -\liminf_{n\to\infty}\int_{\Omega}b(x,u_n)u_n\, dx\\
 &= \int_{\Omega}F\cdot \nabla u\,dx-\liminf_{n\to\infty} \int_{\Omega}b(x,u_n)u_n\,dx.
 \end{split}
 \end{align*}
According to the fact that $\nabla T_k(u) \to \nabla u$ modularly in $L_M(\Omega)$, we obtain  

\begin{align*}
\begin{split}
 &\int_{\Omega}F\cdot \nabla u\,dx-\liminf_{n\to\infty} \int_{\Omega}b(x,u_n)u_n\,dx=\lim_{k \to \infty} \int_\Omega F\cdot \nabla T_k(u)\,dx-\liminf_{n\to\infty} \int_{\Omega}b(x,u_n)u_n\,dx.
 \end{split}
 \end{align*}
In addition, in the step ~$2$, we have used the approximation theory (lemma~\ref{lem:approximation})  to extend the test function of \eqref{eq:h-for-sobolev2} to hold for any $v\in V^1_0L_M(\Omega)\cap L^\infty(\Omega)$, so taking the test function as $T_k(u)$ in \eqref{eq:h-for-sobolev2} and  combining with Lemma \ref{lem:PHI}, we obtain
 \begin{align*}
   \begin{split}
     &\lim_{k \to \infty} \int_\Omega F\cdot \nabla T_k(u)\,dx-\liminf_{n\to\infty} \int_{\Omega}b(x,u_n)u_n\,dx \\
 &\quad=\lim_{k \to \infty} \int_\Omega h \cdot \nabla T_k(u) +b(x,u)T_{k}(u)+\Phi(u)\nabla T_{k}(u)\,dx-\liminf_{n\to\infty} \int_{\Omega}b(x,u_n)u_n\,dx\\
 &\quad =\int_{\Omega} h \cdot \nabla u \, dx+\int_{\Omega}b(x,u)u\,dx-\liminf_{n\to\infty} \int_{\Omega}b(x,u_n)u_n\,dx\\
 &\quad \leq \int_{\Omega} h \cdot \nabla u \, dx.
   \end{split}  
 \end{align*}
Therefore, we deduce that
\begin{eqnarray}\label{eq:I_1}
    I_1=\limsup_{n \to \infty} \int_\Omega \cA(x,\nabla u_n) \cdot \nabla u_n \, dx\leq \int_{\Omega} h \cdot \nabla u \, dx.
\end{eqnarray}
For $I_2$, recalling the fact ~\eqref{limA}, it is easy to check that 
\begin{equation}\label{eq:I_2}
  \ I_2=\limsup_{n \to \infty} \int_\Omega -\cA(x,\nabla u_n)\cdot w \, dx =-\int_\Omega h \cdot w \, dx.
\end{equation}
In addition, according to the fact that $\cA(\cdot,w) \in L^\infty(\Omega,\R^d)$, we have
\begin{equation}\label{eq:I_3}
 I_3 = \limsup_{n \to \infty}\int_\Omega -\cA(x,w) \cdot \nabla u_n \, dx = -\int_\Omega \cA(x,w) \cdot \nabla u \, dx.
\end{equation}
Lastly, 
$I_4$ does not depend on $n$ or $k$.

All in all, combining \eqref{eq:I_1}--\eqref{eq:I_3} we see that
\begin{align*}
    0 \leq \int_\Omega (h - \cA(x,w)) \cdot (\nabla u - w) \, dx.
\end{align*}
Now the monotonicity trick (Lemma~\ref{lem:mon-trick}) yields that $h = \cA(x,\nabla u)$ almost everywhere in $\Omega$. Plugging this to \eqref{eq:h-for-sobolev2} we see that $u_n$ is a weak solution to equation \eqref{eq:main}.\newline

\textbf{Step 4} (uniqueness): We prove uniqueness of the weak solution under the additional assumptions that $s \mapsto b(x,s)$ is strictly increasing for almost every $x \in \Omega$ and $\Phi$ is Lipschitz continuous. We take an approximation of a Heaviside function and pass to the limit, see \cite{uniq,GWWZ-2012,WZ-2010} for similar proofs. 

Assume that there exists two weak solutions $u^1$ and $u^2$. Let us take an approximation of a Heaviside function $H(t) = \chi_{\{t \geq 0\}}$ given by
\begin{align*}
    H_\delta (t): =
    \begin{cases}
    0, &\  t <0, \\
    \tfrac{1}{\delta} t,& \ 0\leq t \leq \delta, \\
    1,& \  t>\delta.
    \end{cases}
\end{align*}
Let $k \in \N$ be arbitrary and test the equation \eqref{eq:main} with $\varphi = H_\delta (u^1-u^2) \in V^1_0L_M(\Omega) \cap L^\infty(\Omega)$ to get
\begin{equation*}
    \int_\Omega \big(\cA(x,\nabla u^i) + \Phi(u^i) \big)\cdot \nabla \varphi + b(x,u^i) \varphi \, dx = \int_\Omega F \cdot \nabla \varphi \, dx,\quad i=1,2\,.
\end{equation*} 
Subtracting the equations we receive
\begin{align*}
    J_1+J_2+J_3&:=\frac{1}{\delta} \int_{\{0\leq u^1-u^2 \leq \delta\}} \big(\cA(x,\nabla u^1) - \cA(x,\nabla u^2)\big) \cdot \nabla (u^1-u^2) \, dx \\
    &\quad + \frac{1}{\delta} \int_{\{0\leq u^1-u^2 \leq \delta\}} \big(\Phi(u^1)-\Phi(u^2)\big) \cdot \nabla (u^1-u^2) \, dx \\
    &\quad + \int_\Omega (b(x,u^1)-b(x,u^2))H_\delta(u^1-u^2) \, dx = 0.
\end{align*}

We note that $J_1 \geq 0$ for each $\delta$ due to (A3). For $J_2$ we estimate with Lipschitz continuity of $\Phi$
\begin{align*}
    J_2 &\leq \frac{1}{\delta} \int_{\{0\leq u^1-u^2 \leq \delta\}} |\Phi(u^1)-\Phi(u^2)| |\nabla (u^1-u^2)| \, dx
    \leq \frac{1}{\delta} \int_{\{0\leq u^1-u^2 \leq \delta\}} L_\Phi \delta |\nabla (u^1-u^2)| \, dx \to 0
\end{align*}
as $\delta \to 0$ as we integrate over a shrinking domain. Lastly we note that due to Lemma~\ref{lem:weak-conv} we have
\begin{align*}
    \lim_{\delta \to 0} J_3 = \int_\Omega (b(x,u^1)-b(x,u^2))H(u^1-u^2) \, dx.
\end{align*}
Dropping the non-negative $J_1$ and taking into account limits of $J_2$ and $J_3$ we see that 
\begin{align*}
    \int_\Omega (b(x,u^1)-b(x,u^2))H(u^1-u^2) \, dx \leq 0
\end{align*}
so that $(b(x,u^1)-b(x,u^2))H(u^1-u^2) =0$ almost everywhere due to sign condition of $b$. 
Since $b$ is strictly increasing with respect to the second variable we see that $u^1\leq u^2$. Considering $\varphi = H_\delta(u^2-u^1)$ yields the opposite inequality and thus $u^1=u^2$.
\end{proof}

\printbibliography{}
\end{document}